\documentclass[a4paper,10pt]{article}
\usepackage[latin1]{inputenc}
\usepackage[T1]{fontenc}
\usepackage{amsmath}
\usepackage{amsfonts}
\usepackage{amstext}
\usepackage{amsthm}
\usepackage[all]{xy}
\usepackage{geometry}
\usepackage{srcltx}
\usepackage{graphics}
\usepackage{epsfig}
\usepackage{subfigure}
\usepackage{slashed}
\usepackage{color}
\usepackage{amssymb}
\usepackage{srcltx}
\usepackage{cleveref}
\newtheorem{theorem}{Theorem}[section]
\newtheorem{lemma}[theorem]{Lemma}

\newtheorem{cor}[theorem]{Corollary}

\DeclareMathOperator{\sfl}{sf}
\DeclareMathOperator{\diag}{diag}

\DeclareMathOperator{\sgn}{sgn}

\DeclareMathOperator{\Mat}{Mat}

\title{Bifurcation for a Class of Indefinite Elliptic Systems by Comparison Theory for the Spectral Flow via an Index Theorem}
 
\author{Joanna Janczewska, Melanie M\"ockel and Nils Waterstraat}

\begin{document}
	\date{}
	\maketitle
	
	\begin{abstract}\noindent
We consider families of strongly indefinite systems of elliptic PDE and investigate bifurcation from a trivial branch of solutions by using the spectral flow. The novelty in our approach is a refined way to apply a comparison principle which is based on an index theorem for a certain class of Fredholm operators that is of independent interest. Finally, we use our findings for a bifurcation problem on shrinking domains that originates from works of Morse and Smale.  		
	\end{abstract}

	\section{Introduction}
	Let $U\subset \mathbb{R}^N$ be a bounded smooth domain for some $N\in \mathbb{N}$ and consider the system of elliptic partial differential equations
	\begin{align}
		\label{eq:pdeGeneral}
		\left\{
		\begin{array}{rl}
			A\Delta u(x) &= \nabla_u F(\lambda,x,u(x))\ \hspace*{0.25cm} \mathrm{in}\ U\\		
			u(x) &= 0\ \hspace*{2.5cm}  \mathrm{on}\ \partial U,
		\end{array}
		\right.
	\end{align}
	where $\lambda\in[0,1]$, $A := \diag\{a_1,..,a_p\}\in\Mat(p,\mathbb{R})$, $a_i\in\{\pm 1\}$, $i=1,\ldots,p$ and $F:[0,1]\times U\times\mathbb{R}^p\rightarrow\mathbb{R}$ is a $C^2$-map. We assume that 
	
	\[\nabla_u F(\lambda,x,0)=0,\]
which implies that $u\equiv 0$ is a solution of \eqref{eq:pdeGeneral} for all $\lambda$, and our aim is to find novel criteria for bifurcation from this trivial branch of solutions. Bifurcation from a trivial branch for strongly indefinite elliptic systems as \eqref{eq:pdeGeneral} has been studied by various authors and different methods, e.g., \cite{GawryRy}, \cite{GoleRy}, \cite{AleIchDomain}, \cite{Szulkin}. Our approach is based on the spectral flow and a comparison property of it.\\
The spectral flow is a homotopy invariant for paths of selfadjoint Fredholm operators that was introduced by Atiyah, Patodi and Singer in \cite{AtiyahPatodi}. It has found several applications in symplectic analysis and mathematical physics (see \cite{NoraHermannNils}, \cite{Fredholm}). In \cite{Specflow} Fitzpatrick, Pejsachowicz and Recht introduced an application in variational bifurcation theory which has generalised several previously well known theorems in this field. The already mentioned comparison property of the spectral flow was originally shown by Pejsachowicz and the last named author in \cite{BifJac}. It allows to find bifurcation points by comparing a given system like \eqref{eq:pdeGeneral} to a simpler one, which here will be an equation as \eqref{eq:pdeGeneral} that does no longer explicitly depend on $x\in U$, i.e.,

	\begin{align}
		\label{eq:pdeIntro}
		\left\{
		\begin{array}{rl}
			A\Delta u(x) &= \nabla_u F(\lambda,u(x))\ \hspace*{0.25cm} \mathrm{in}\ U\\		
			u &= 0\ \hspace*{1.65cm}  \hspace*{0.25cm}\mathrm{on}\ \partial U,
		\end{array}
		\right.
	\end{align}
for some $C^2$-map $F:[0,1]\times\mathbb{R}^p\rightarrow\mathbb{R}$.\\
Solutions of equation \eqref{eq:pdeGeneral} (and thus in particular of \eqref{eq:pdeIntro}) can be obtained for any fixed $\lambda\in[0,1]$ as critical points of a $C^2$- functional under common growth assumptions. The Hessians at the trivial solution $u\equiv 0$ are selfadjoint Fredholm operators, and by the already mentioned theorem by Fitzpatrick, Pejsachowicz and Recht \cite{Specflow}, a non-vanishing spectral flow of this path yields the existence of a bifurcation point. The first main theorem of this paper is an index theorem that yields an explicit formula for the spectral flow of the Hessians coming from \eqref{eq:pdeIntro} and is of independent interest. As a consequence, we obtain a bifurcation criterion for equations of the type \eqref{eq:pdeIntro}, which relates the Dirichlet eigenvalues of the domain $U$ to the existence of bifurcation points for \eqref{eq:pdeIntro}. Let us note that the impact of the Dirichlet eigenvalues of a smooth domain for other types of strongly indefinite equations has been studied by Szulkin in \cite{Szulkin} and by the last named author in \cite{Edinburgh}. Afterwards we consider the more general equations \eqref{eq:pdeGeneral} and obtain bifurcation criteria by comparing them to a system of the type \eqref{eq:pdeIntro} which does no longer explicitly depend on $x\in U$. The novelty in our approach is that we use a refined way of applying the comparison principle from \cite{BifJac}, which eventually is based on Weyl's inequalities for eigenvalues. The final section deals with an application of the obtained spectral flow formulas to solutions of elliptic PDEs on shrinking domains, which is a setting that has previously been studied, e.g., by Smale in \cite{Smale}, by Deng and Jones in \cite{DengJones}, by Cox, Jones, Latushkin and Sukhtayev in \cite{Cox}, and by the last author and Portaluri in \cite{AleIchDomain}, \cite{AleBall} and \cite{AleSmaleIndef}. The latter works deal with bifurcation and in \cite{AleSmaleIndef} strongly indefinite systems of the type \eqref{eq:pdeGeneral} are considered, where the right hand side does not depend on a parameter $\lambda$, but the domain $U$ is assumed to be star-shaped and gets shrunk to a point. The main result of \cite{AleSmaleIndef} shows the appearance of non-trivial solutions during the shrinking of the domain (i.e. bifurcation) if a certain Maslov index in an infinite-dimensional symplectic Hilbert space does not vanish, which is theoretically appealing but not very useful in practice. As an application of the obtained results of this paper, we underpin their strength by showing that in this setting bifurcation can be found just from the coefficients of the corresponding linearised equations.\\
The paper is structured as follows. We firstly survey in the next section the spectral flow and its application to bifurcation theory as well as its comparison property. Section 3 briefly recaps the variational formulation of \eqref{eq:pdeGeneral}. In Section 4 we prove the announced index theorem and explain its application to the existence of bifurcation for \eqref{eq:pdeIntro}. Section 5 deals with the refined comparison method and also contains an example that shows that the obtained result is stronger than what could have been obtained by previously known applications of the comparison property from \cite{BifJac}. The final Section 6 discusses the announced applications of the obtained results to bifurcation on shrinking domains.

%%%%%%%%%%%%%%%%%%%%%%%%%%%%%%%%%%%%%%%%%%%%%%%%%%%%%%%%%%%%%%%%%%%%%%%%%%%%%%%%%%%%%%%%%%%%%%%%%%%%%%%%%%%%%%%%%%%%%%%%%%%%%%%%%%%%%%%%%%%%%%%%%%%%%%%%%%%%%%%%%%%%%%%

	\section{Recap: The Spectral Flow and Bifurcation}\label{section:sfl}
	The aim of this section is to briefly recall the spectral flow for paths of selfadjoint Fredholm operators and its application in bifurcation theory.\\
	Let $H$ be a real separable Hilbert space and $\mathcal{L}(H)$ the Banach space of all bounded linear operators on $H$. A selfadjoint operator in $\mathcal{L}(H)$ is Fredholm if its range is closed and its kernel is of finite dimension. In what follows, we denote by $\Phi_S(H)\subset\mathcal{L}(H)$ the set of all selfadjoint Fredholm operators. It was shown by Atiyah and Singer in \cite{AtiyahSinger} that $\Phi_S(H)$ has three connected components if $H$ is of infinite dimension. If we denote by $\sigma_{ess}(T)=\{\lambda\in\mathbb{R}: \lambda-T\notin\Phi_S(H)\}$ the \textit{essential spectrum} of an operator $T\in\Phi_S(H)$, then these components are
	
	\begin{align*}
		\Phi^+_S(H)&=\{T\in\Phi_S(H):\, \sigma_{ess}(T)\subset (0,\infty)\},\\
		\Phi^-_S(H)&=\{T\in\Phi_S(H):\, \sigma_{ess}(T)\subset (-\infty,0)\},\\
		\Phi^i_S(H)&=\Phi_S(H)\setminus (\Phi^+_S(H)\cup \Phi^-_S(H)).
	\end{align*}  
	It is not difficult to see that $\Phi^\pm_S(H)$ are contractible as topological spaces. Atiyah and Singer showed that instead $\Phi^i_S(H)$ has a non-trivial topology. In particular, they constructed in \cite{AtiyahPatodi} an explicit isomorphism $\sfl:\pi_1(\Phi^i_S(H))\rightarrow\mathbb{Z}$ which they called the spectral flow. Here $\pi_1(\Phi^i_S(H))$ is the fundamental group of $\Phi^i_S(H)$ which consists of homotopy classes of based loops in $\Phi^i_S(H)$. Roughly speaking, the spectral flow counts the net number of eigenvalues crossing $0$ in the course of the path. The definition was later extended to open paths in $\Phi_S(H)$ and the spectral flow became a frequently used invariant in various branches of mathematics. Here we refrain from recalling its definition, which can be found, e.g., in \cite{Phillips}, \cite{Robbin-Salamon} or \cite{Specflow}. Instead we introduce it by the following theorem of Ciriza, Fitzpatrick and Pejsachowicz \cite{JacoboUniqueness} (cf. \cite{Lesch}, \cite{CompSfl}), where $\Omega(\Phi_S(H),G\Phi_S(H))$ denotes the set of paths in $\Phi_S(H)$ with invertible endpoints.
	
	\begin{theorem}
		For each real separable Hilbert space $H$ there is one and only one map
		
		\[\sfl:\Omega(\Phi_S(H),G\Phi_S(H))\rightarrow\mathbb{Z}\]
		such that
		
		\begin{itemize}
			\item[(H)] If $h:I\times I\rightarrow\Phi_S(H)$ is continuous and $h(s,0), h(s,1)$ are invertible for all $s\in I$, then
			
			\[\sfl(h(0,\cdot)=\sfl(h(1,\cdot)).\]
			\item[(Z)] If $L=\{L_\lambda\}_{\lambda\in I}\in\Omega(\Phi_S(H),G\Phi_S(H))$ is such that $L_\lambda$ is invertible for all $\lambda\in I$, then $\sfl(L)=0$.
			\item[(S)] If $H'$ is another real separable Hilbert space, $L\in \Omega(\Phi_S(H),G\Phi_S(H))$ and\linebreak $L'\in \Omega(\Phi_S(H'),G\Phi_S(H'))$, then
			
			\[\sfl(L\oplus L')=\sfl(L)+\sfl(L').\]
			\item[(N)] If $L=\{L_\lambda\}_{\lambda\in I}\in\Omega(\Phi_S(H),G\Phi_S(H))$ and $H$ is of finite dimension, then 
			
			\begin{align}\label{sfl=Morse}
			\sfl(L)=\mu_{Morse}({L_0})-\mu_{Morse}({L_1}).
			\end{align} 
		\end{itemize}
	\end{theorem} 
\noindent	Here $\mu_{Morse}(L_\lambda)$ denotes the Morse index, i.e., the number of negative eigenvalues counted according to their multiplicities. Let us note that this number is also finite for operators in $\Phi^+_S(H)$ if $H$ is of infinite dimension, and \eqref{sfl=Morse} still holds for paths in this component of $\Phi_S(H)$. Further properties of the spectral flow and methods to compute it can be found, e.g., in \cite{Specflow} or \cite{NoraHermannNils}. As we do not need them here, we refrain from recalling them. Let us finally note that the assumption of invertibility of the endpoints can easily be lifted by a small perturbation of the path. The spectral flow is then still homotopy invariant as long as the homotopy keeps the endpoints of the path fixed. Details can be found in \cite[\S 7]{BifJac}.\\
Let us now consider a continuous family of functionals $f:I\times H\rightarrow\mathbb{R}$ such that each $f_\lambda:=f(\lambda,\cdot):H\rightarrow\mathbb{R}$ is $C^2$ and its derivatives depend continuously on the parameter $\lambda\in I$. Henceforth we assume that $0\in H$ is a critical point of all $f_\lambda$ and $L_\lambda:=\nabla^2_uf_\lambda(0)$, the Hessian of $f_\lambda$ at the critical point $0\in H$, is a Fredholm operator. Then $L:=\{L_\lambda\}_{\lambda\in I}$ is a path in $\Phi_S(H)$ and thus the spectral flow $\sfl(L)$ is defined. A well known theorem in nonlinear analysis asserts that if $L_0$, $L_1$ are invertible and their Morse indices are finite and differ, then there is a bifurcation of critical points of the family $f$ from $0\in H$ (cf., e.g., \cite{Mawhin}, \cite{SmollerWasserman}). By \eqref{sfl=Morse} and the explanations below that equation, this just means that the spectral flow of the path $L$ of Hessians is non-trivial. The following generalisation to functionals $f$ having Hessians in an arbitrary component of $\Phi_S(H)$ was obtained by Fitzpatrick, Pejsachowicz and Recht in \cite{Specflow} (cf. \cite{BifJac}).
	
	\begin{theorem}\label{thm-bif}
		If $L_0$ and $L_1$ are invertible and $\sfl(L)\neq 0$, then there is a bifurcation of critical points of $f$ from the trivial branach, i.e., there exists $\lambda^\ast\in(0,1)$ such that in every neighbourhood $U$ of $(\lambda^\ast,0)$ in $I\times H$ there is some $(\lambda,u)\in U$ such that $u\neq 0$ is a critical point of $f_\lambda$.  
	\end{theorem}   
\noindent
Let us note that Alexander and Fitzpatrick have shown in \cite{AlexanderMike} that this theorem is optimal in a reasonable sense. On the other hand, the spectral flow often turns out to be difficult to compute which limits the applicability of Theorem \ref{thm-bif}. However, note that Theorem \ref{thm-bif} only requires the spectral flow to be non-trivial and thus the exact quantity actually is not needed. In \cite{BifJac} the second author obtained in a joint work with Pejsachowicz the following way to estimate the spectral flow.
	
	\begin{theorem}\label{thm-comparison}
		Let $L=\{L_\lambda\}_{\lambda\in I}$ and $M=\{M_\lambda\}_{\lambda\in I}$ be paths in $\Phi_S(H)$ such that $L_\lambda-M_\lambda$ is compact for all $\lambda\in I$. If
		
		\[L_0\geq M_0\quad\text{and}\quad L_1\leq M_1,\]
		then
		
		\[\sfl(L)\leq\sfl(M).\] 
	\end{theorem} 
\noindent	Here $\leq$ denotes the common partial order on the set of all selfadjoint operators given by
	
	\[T\leq S :\Longleftrightarrow \langle(S-T)u,u\rangle\geq 0,\, u\in H.\] 
One of the main results of this work, Theorem \ref{thm-BifComp}, is a refined version of the previous theorem that applies to operators stemming from equations of the type \eqref{eq:pdeGeneral}.

%	Let us finally point out that the exact value of the spectral flow in Theorem \ref{thm-bif} can imply a lower bound on the number of bifurcation points of $f$ if there is an a-priori bound on the dimension of the kernels of the operators $L_\lambda$ (cf. \cite{BifJac}).  

%%%%%%%%%%%%%%%%%%%%%%%%%%%%%%%%%%%%%%%%%%%%%%%%%%%%%%%%%%%%%%%%%%%%%%%%%%%%%%%%%%%%%%%%%%%%%%%%%%%%%%%%%%%%%%%%%%%%%%%%%%%%%%%%%%%%%%%%%%%%%%%%%%%%%%%%%%%%%%%%%%%%%%%
	
	\section{The Variational Setting}\label{sec:VarSet}
	The aim of this section is to briefly recap the variational setting of the equations  \eqref{eq:pdeGeneral}, where we mainly follow \cite{GoleRy}. Henceforth $I:=[0,1]$ denotes the unit interval and we assume that
	\begin{itemize}
		\item[(A1)] $F\in \mathcal{C}^2(I\times U\times \mathbb{R}^p,\mathbb{R})$,
		
		\item[(A2)] $0$ is a critical point of $F_\lambda:=F(\lambda,x,\cdot):\mathbb{R}^p\rightarrow\mathbb{R}$ for all $(\lambda,x) \in I\times U$. In what follows, we set  $$B_\lambda(x):=\nabla^2_uF(\lambda,x,0)\in\Mat(p,\mathbb{R}) \,.$$
		
		\item[(A3)] There exist $C>0$ and $1\le s< (N+2)(N-2)^{-1}$ if $N\geq 3$ such that 
		$$\vert\nabla^2_u F(\lambda,x,u)\vert\le C(1+\vert u\vert^{s-1})\,.$$ 
		If $N=2$, we instead require that $s\in [1,\infty)$, and for $N=1$ we do not impose any growth condition on $F$.
		
		\item[(A4)] We let $p_1\in\mathbb{N}\cup\{0\}$, $0\leq p_1\leq p$,  $$a_1=...=a_{p_1}=-1,\, a_{p_1+1}=...=a_p=1,$$
		and set $A = \diag \{\underset{p_1}{\underbrace{-1,...,-1}},\underset{p_2:=p-p_1}{\underbrace{1,...,1}}\}$.
	\end{itemize}
	Let $H^1_0(U)$ be the standard Sobolev space with scalar product 
	$$\left\langle u_1,v_1\right\rangle _{H^1_0(U)} =\int_U \left\langle \nabla u_1(x),\nabla v_1(x)\right\rangle\,\text{d}x.$$
	Then $H:=\bigoplus_{i=1}^p H^1_0(U)$ is a Hilbert space with respect to 
	$$\left\langle u,v\right\rangle _H := \sum_{i=1}^{p}\left\langle u_i,v_i\right\rangle _{H^1_0(U)}.$$
	Now consider the map $f:I\times H\to \mathbb{R}$  given by
	\begin{align}
		\label{eq:functional}
		f(\lambda,u) :=\frac{1}{2}\int_U \sum_{i=1}^p (-a_i\vert\nabla u_i(x)\vert^2)\,\text{d}x-\int_U F(\lambda,x, u(x))\,\text{d}x.
	\end{align}
	It follows from assumption $(A2)$ that there exists a map $g\in \mathcal{C}^2(I\times U\times \mathbb{R}^p,\mathbb{R})$ such that 
	$$F(\lambda,x, u)=\frac{1}{2}\left\langle B_\lambda(x) u,u\right\rangle +g(\lambda,x,u)$$ 
	and for every $\lambda\in I$, $x\in U$, we have $\nabla_u\,g(\lambda,x,0)=0$ as well as $\nabla_u^2\,g(\lambda,x,0)=0$.\\
	It is well known (cf. \cite{Rabinowitz}) that $f$ is in $\mathcal{C}^2(I\times H,\mathbb{R})$ under the assumptions $(A1)-(A4)$ and the gradient of $f(\lambda,\cdot):H\rightarrow\mathbb{R}$ is of the form  $$\nabla_uf(\lambda,u) = Tu+K_{\lambda}u-\nabla_u\eta(\lambda,u),$$
		where 
		\begin{itemize}
			\item[(i)] $T:H\to H$ is the selfadjoint invertible operator 
			\begin{align}
				\label{eq:OperatorT}
				Tu:=-Au.	
			\end{align}
			\item [(ii)] $K_{\lambda}:H\to H$ implicitly is given by
			\begin{align}
				\label{eq:OperatorK}
				\left\langle K_{\lambda}u,v\right\rangle _H=-\int_U \left\langle  B_\lambda(x) u(x),v(x)\right\rangle_{\mathbb{R}^p} \,\text{d}x	
			\end{align} 
			and it is a selfadjoint compact operator.
			\item[(iii)] $\eta:I\times H \to \mathbb{R}$ is the $C^2$-map defined by $$\eta(\lambda,u)=\int_U g(\lambda,x, u(x))\,\text{d}x,$$ 
			and $\nabla_u\eta(\lambda,0)=0$ as well as $\nabla_u^2\eta(\lambda,0)=0$ for all $\lambda \in I$.
		\end{itemize}
		Moreover, the critical points of $f_\lambda:=f(\lambda,\cdot):H\rightarrow\mathbb{R}$ are the weak solutions of \eqref{eq:pdeGeneral}, and thus in particular $0\in  H$ is a critical point of all $f_\lambda$, $\lambda\in I$.
%	\begin{proof}
%	
%		We first note that
%		\begin{align*}
%			f(\lambda,u)
%			=&\frac{1}{2}\int_U \sum_{i=1}^p (-a_i\vert \nabla u_i(x)\vert^2)\,\text{d}x-\int_u F(\lambda,x, u(x)) \,\text{d}x\\
%			=& \frac{1}{2}\int_U \sum_{i=1}^p (-a_i\vert \nabla u_i(x)\vert ^2) \,\text{d}x -\int_U \frac{1}{2}\left\langle B_\lambda (x)u(x),u(x)\right\rangle +g(\lambda,u(x))\,\text{d}x\\=&\frac{1}{2}\left\langle Tu,u\right\rangle _H+\left\langle K_\lambda u,u\right\rangle _H-\eta(\lambda, u) 
%		\end{align*}
%		and thus  
%		\begin{align*}
%			\left\langle \nabla_u f(\lambda, u),v\right\rangle _H
%			=&\left\langle Tu,v\right\rangle _H + \left\langle K_\lambda u,v\right\rangle _H- \left\langle \nabla \eta(\lambda,u),v\right\rangle _H,\quad v\in H.
%		\end{align*}
	%	The properties in $(i)$ and $(ii)$ follow from the definition of the operators $T$ and $K_\lambda$. To prove that the operator in $%(iii)$ is a $C^2$-map, see \cite{Rabinowitz}.
%	\end{proof}
	\noindent
	It follows from $(i)-(iii)$ that the Hessians $\nabla_u^2f_\lambda(0)$ at the critical point $0\in H$ are the selfadjoint operators 
	\begin{align}
		\label{eq:Fredholmop}
		L_\lambda:=T+K_{\lambda}
	\end{align} 
	Note that these are compact perturbations of an invertible operator and hence Fredholm. Accordingly $L:=\{L_\lambda \}_{\lambda\in I}$ is a path of selfadjoint Fredholm operators so that the spectral flow of $L$ is defined. Moreover, the kernel of $L_\lambda$ consists of the solutions of the linearised equations
	
	\begin{align}
		\label{eq:pdelinearised}
		\left\{
		\begin{array}{rl}
			A\Delta u(x) &= B_\lambda(x) u(x)\ \hspace*{0.25cm} \mathrm{in}\ U\\		
			u &= 0\ \hspace*{1.65cm}  \mathrm{on}\ \partial U,
		\end{array}
		\right.
	\end{align}
	and thus $L_\lambda$ is invertible if and only if \eqref{eq:pdelinearised} has no non-trivial solution.
		
%%%%%%%%%%%%%%%%%%%%%%%%%%%%%%%%%%%%%%%%%%%%%%%%%%%%%%%%%%%%%%%%%%%%%%%%%%%%%%%%%%%%%%%%%%%%%%%%%%%%%%%%%%%%%%%%%%%%%%%%%%%%%%%%%%%%%%%%%%%%%%%%%%%%%%%%%%%%%%%%%%%%%%%
	
	\section{Index Theorem and Bifurcation for \eqref{eq:pdeIntro}}\label{sec:4}
	In this section we consider the system \eqref{eq:pdeIntro} and thus assume in addition to $(A1)-(A4)$ that the matrix family $B$ in $(A2)$ does not depend on $x\in U$, i.e., \eqref{eq:pdelinearised} is now of the simpler form
	
	\begin{align}
		\label{eq:pdelinearised2}
		\left\{
		\begin{array}{rl}
			A\Delta u(x) &= B_\lambda u(x)\ \hspace*{0.25cm} \mathrm{in}\ U\\		
			u &= 0\ \hspace*{1.00cm}  \mathrm{on}\ \partial U.
		\end{array}
		\right.
	\end{align}
	Let us emphasise that all results of this section obviously also apply to equations of the more general type \eqref{eq:pdeGeneral} as long as this additional condition on $B$ is satisfied.\\
	 In what follows we let $\{f_n\}_{n\in\mathbb{N}}$ be the orthonormal basis of $H^1_0(U)$ given by the eigenfunctions of the Dirichlet boundary problem
		
	\begin{align} 
		\label{eq:ewp}
		\left\{
		\begin{array}{rl}-\Delta 
			u&=\,\lambda u\, \hspace{0.25cm}  \mathrm{in}\ U\\
			u&=\,0\ \hspace{0.45cm} \mathrm{on}\ \partial U.
		\end{array}
		\right.
	\end{align}
	and we assume that the corresponding eigenvalues $\alpha_n$ satisfy $0<\alpha_n\leq\alpha_m$ if $m\geq n$. It is readily seen from Green's identities that for all $n,m\in \mathbb{N},\,  m\not=n$ 
	\begin{align} 	
		\label{eq:onb1}
		\int_U f_nf_m\,\text{d}x=0
	\end{align} 
	and 
	\begin{align*} 
		\label{eq:onb2}
		1=\int_U \vert\nabla f_n\vert^2\,\text{d}x=\alpha_n\int_Uf^2_n\,\text{d}x.
	\end{align*}
	We now consider the orthonormal decomposition 
	\begin{align}
		H=\bigoplus_{i=1}^p H^1_0(U) =\bigoplus_{k\in \mathbb{N}} H_k
	\end{align}
	where $H_k:=\text{span}\{f_k\cdot e_i\mid 1\le i\le p \}$ and   $\{e_i\}_{i=1}^p$ denotes the standard basis of $\mathbb{R}^p$.
	\begin{lemma} 
		\label{lemma:OGProj}
		Let $P_k$, $k\in\mathbb{N}$, be the orthogonal projection in $H$ onto $H_k$. If $k,l\in \mathbb{N}, k\not=l$, then $$P_kK_{\lambda}P_l=0.$$
	\end{lemma}
	\begin{proof}
		Clearly,
		$$\left\langle P_kK_{\lambda}P_lu,v\right\rangle _H=\left\langle K_{\lambda}P_lu,P_kv\right\rangle_H, \quad u,v\in H.$$
		For $u=f_ke_i$ and $v=f_le_j$, where $1\leq i,j\leq p$, it follows by \eqref{eq:onb1} that 
		$$\left\langle K_{\lambda}u,v\right\rangle_H=\int_Uf_kf_l \langle B_\lambda e_i,e_j\rangle\,\text{d}x=\langle B_\lambda e_i,e_j\rangle\int_Uf_kf_l \,\text{d}x=0,$$
		which shows that $P_kK_{\lambda}P_l=0$ for $ k\not=l$.
	\end{proof}
	\noindent
	By the previous lemma, the operators $K_{\lambda}$ leave the spaces $H_k$ invariant. Note that this also is the case for the operator $T$, and thus we obtain paths of operators by $L_\lambda^k :=L_\lambda\mid_{H_k}:H_k\rightarrow H_k$. It is readily seen that $L_\lambda^k$ is represented by the matrix
	\begin{align*}
		L_\lambda^k = 
		\begin{pmatrix}
			I_{p_1\times p_1} &  & 0 \\
			0 &  & -I_{p_2\times p_2}
		\end{pmatrix}
		-\frac{1}{\alpha_k}B_{\lambda}=-A-\frac{1}{\alpha_k}B_{\lambda},\, \lambda \in I,
	\end{align*}
	with respect to the orthonormal basis $\{f_ke_i\mid 1\le i\le p\}$ of $H_k$.
	Since $\alpha_k\to \infty$ as $k\to \infty$, there exists $n\in\mathbb{N}$ such that 
	$L_\lambda^k$ is an isomorphism and 
	\begin{align}
		\label{eq:sgn}
		\sgn(L_\lambda^k)=\sgn(-A) \text{ for all }k\ge n,\, \lambda \in I.
	\end{align} 
	Thus the index
	$$i(B_\lambda):=\frac{1}{2}\sum_{k=1}^\infty (\sgn(L^k_\lambda)-\sgn(-A))$$
	of the matrices $B_\lambda$ is well defined. We can now state the main theorem of this section, which expresses the spectral flow of the paths of Hessians $L$ in \eqref{eq:Fredholmop} in terms of this index.
	\begin{theorem}
		\label{thm:bifurctaion1}
		If the assumptions (A1)-(A4) hold and \eqref{eq:pdelinearised2} only has the trivial solution for $\lambda=0$ and $\lambda=1$, then the family of Hessians $L=\{L_\lambda\}_{\lambda\in I}$ satisfies
		
		\[\sfl(L)=i(B_1)-i(B_0).\]
	\end{theorem} 
	\begin{proof}
		Let $n\in \mathbb{N}$ and $H=H^1\oplus H^2$ where  $H^1:=\oplus_{k=1}^n H_k$ and $H^2:=(H^1)^\perp$. Let $P_k$, $k\in\mathbb{N}$, be the orthogonal projection in $H$ onto $H_k$. Then $Q_n:=\sum_{k=1}^n P_k$ is the orthogonal projection onto $H^1$ and $Q_n^\perp:=(I_H-Q_n)$ the orthogonal projection onto $H^2$.
		It follows from \Cref{lemma:OGProj} and \eqref{eq:Fredholmop} that
		\begin{align}
		\label{eq:proj}
		\begin{split}
		L_\lambda 
		&= T +Q_n K_{\lambda} Q_n + Q_n K_{\lambda} Q_n^\perp +Q_n^\perp K_{\lambda} Q_n + Q_n^\perp K_{\lambda} Q_n^\perp \\
		&=T+Q_n K_{\lambda} Q_n +Q_n^\perp K_{\lambda} Q_n^\perp.
		\end{split}				
		\end{align}
		From the compactness of $K_\lambda$, the strong convergence of $Q_n^\perp $ to $0$ for $n\to \infty$ and since $\Vert Q_n^\perp \Vert =1$, we see that 
		\begin{align}
		\label{eq:strconv}
		\Vert  Q_n^\perp K_{\lambda} Q_n^\perp \Vert \to 0, \  n\to \infty,
		\end{align} 
		uniformly in $\lambda$. 
		According to the assumptions, the operator $L_\lambda$ is invertible for $\lambda=0,1$. Thus there is $C>0$ such that 
		$$ \Vert L_0u\Vert \ge 3C \Vert u \Vert,\ u\in H,$$ and 
		$$\Vert L_1u\Vert \ge 3C\Vert u\Vert,\  u\in H.$$
		We now let $n_0$ be sufficiently large such that 
		
		\begin{align}
			\label{eq:abschaetzung2}
			\Vert  Q_n^\perp K_{\lambda} Q_n^\perp \Vert<C,\quad n\geq n_0.
		\end{align}
		It follows from \eqref{eq:proj} and \eqref{eq:strconv} that 
		\begin{align}
			\label{eq: abschaetzung1}
			\Vert Tu+ Q_n K_{\lambda} Q_n u \Vert \ge 2C \Vert u\Vert ,\ u\in H,\,\, \lambda=0,1,\quad n\geq n_0.
		\end{align}
		Henceforth we assume that $n_0$ is sufficiently large such that in addition to \eqref{eq: abschaetzung1} also $L_\lambda^k$ is invertible and \eqref{eq:sgn} holds for all $k\ge n_0$.\\ 
		We now consider for some $n\ge n_0$ the homotopy  
		$$h:I\times I\to \Phi_S(H)$$
		defined by  
		$$ h(t,\lambda)=T+Q_nK_{\lambda}Q_n+tQ_n^\perp K_{\lambda}Q_n^\perp.$$
		By \eqref{eq: abschaetzung1} and \eqref{eq:abschaetzung2}, we conclude that for $\lambda =0,1$  $$\Vert h(t,\lambda)u\Vert \ge C \Vert u \Vert,\ u\in H,$$ and thus $h(t,0)$ and  $h(t,1)$ are invertible for all $t\in I$. Consequently, by the properties of the spectral flow from Section \ref{section:sfl}, 
		\begin{align}
			\label{eq:sf1}
			\begin{split}
			\text{sf}(L)
			&{=}\text{sf}(L\mid_{H^1})+\text{sf}(L\mid_{H^2}){=}\text{sf}(L\mid_{H^1})\\
			&{=} \sum_{k=1}^{n}\text{sf}(L\mid_{H_k}){=}\sum_{k=1}^{n}\left(\mu_{\text{Morse}}(L_0\mid_{H_k})-\mu_{\text{Morse}}(L_1\mid_{H_k})\right).
			\end{split}				 
		\end{align}
		The signature and the Morse index of a symmetric $p\times p$-matrix $M$ are related by $$\text{sgn}(M)=p-2\mu_{\text{Morse}}(M).$$ 
		Thus we obtain in \eqref{eq:sf1}
		\begin{align}
			\label{eq:sf2}
			\begin{split}
			&\sum_{k=1}^{n}\left(\mu_{\text{Morse}}(L_0\mid_{H_k})-\mu_{\text{Morse}}(L_1\mid_{H_k})\right)\\
			&=\sum_{k=1}^n (\frac{p}{2}-\frac{1}{2}\sgn(L_0^k)-\frac{p}{2}+\frac{1}{2}\sgn(L_1^k))\\
			&=\frac{1}{2}\sum_{k=1}^n( \sgn(L_1^k)-\sgn(L_0^k)+\sgn(-A)-\sgn(-A))\\
			&=\frac{1}{2}\sum_{k=1}^n(\sgn(L_1^k)-\sgn(-A))-\frac{1}{2}\sum_{k=1}^n(\sgn(L_0^k)-\sgn(-A))\\
			&=i(B_1)-i(B_0),
			\end{split}
		\end{align}
		where we have used \eqref{eq:sgn} in the final equality. Finally the theorem follows from \eqref{eq:sf1} and \eqref{eq:sf2}.
	\end{proof}   
\noindent
From Theorem \ref{thm-bif}, we now immediately obtain the following result.

\begin{cor}\label{cor-bif}
Assume that the assumptions (A1)-(A4) hold for the boundary value problem
		\eqref{eq:pdeIntro}. If the linearised equation \eqref{eq:pdelinearised2} only has the trivial solution for $\lambda=0$ and $\lambda=1$, and 
		$$i(B_0)\not=i(B_1),$$
		then there exists a bifurcation point for \eqref{eq:pdeIntro} in $(0,1)$.
\end{cor}
\noindent 
Let us note that our index $i(B_\lambda)$ and the corresponding bifurcation result in Corollary \ref{cor-bif} are of a similar form as the bifurcation theorem for autonomous Hamiltonian systems by Fitzpatrick, Pejsachowicz and Recht in \cite{SFLPejsachowiczII}. The differences are the appearance of the term $\sgn(-A)$ in the definition of the index as well as most parts of the proof of Theorem \ref{thm:bifurctaion1}, which become necessary when dealing with a PDE instead of an ODE.   
%%%%%%%%%%%%%%%%%%%%%%%%%%%%%%%%%%%%%%%%%%%%%%%%%%%%%%%%%%%%%%%%%%%%%%%%%%%%%%%%%%%%%%%%%%%%%%%%%%%%%%%%%%%%%%%%%%%%%%%%%%%%%%%%%%%%%%%%%%%%%%%%%%%%%%%%%%%%%%%%%%%%%%%
		
	\section{Bifurcation by Comparison for \eqref{eq:pdeGeneral}.}\label{Sec:Comp}
	
We now consider the general system \eqref{eq:pdeGeneral}, where the right hand side explicitly depends on $x\in U$. The aim of this section is to find a criterion for the existence of bifurcation points for \eqref{eq:pdeGeneral} by applying the comparison principle in \Cref{thm-comparison} for the  spectral flow. As introduced in Section 3, $L_\lambda$ denotes the Hessian $\nabla_u^2f_\lambda(0)$ at  the critical point $0\in H$. \newline
Let now $M=\{M_\lambda\}_{\lambda \in I}$ and $N=\{N_\lambda\}_{\lambda \in I}$ be two paths in $\Phi_S(H)$ such that
	$$\left\langle M_\lambda u,v\right\rangle_H:= \left\langle Tu,v\right\rangle _H -\int_U\left\langle C_\lambda u,v\right\rangle \text{d}x$$ and 
	$$\left\langle N_\lambda u,v\right\rangle_H:= \left\langle Tu,v\right\rangle _H -\int_U\left\langle D_\lambda u,v\right\rangle \text{d}x,$$ where $T$ is as in \eqref{eq:OperatorT}, and

	\[C_\lambda:=
		\begin{pmatrix}
		C_{1,\lambda} & 0 \\ 
		0 & C_{2,\lambda}
		\end{pmatrix}\quad\text{and}\quad 
	D_\lambda:=
		\begin{pmatrix}
		D_{1,\lambda} & 0 \\ 
		0 & D_{2,\lambda}
		\end{pmatrix},\]
		for some
		\[ C_{1,\lambda}, D_{1,\lambda}\in\Mat(p_1,\mathbb{R}),\quad C_{2,\lambda}, D_{2,\lambda}\in\Mat(p_2,\mathbb{R}),\]
	are symmetric matrices such that $B_0(x)-C_0,\, C_1-B_1(x)$ and $D_0-B_0(x),\,B_1(x)-D_1$ are positive semi-definite for all $x\in U$. Henceforth, we denote the eigenvalues of a symmetric matrix $M\in \Mat(p,\mathbb{R})$ by $\mu_i(M)$ for $i\in \{1,\dots,p\}$, where we suppose the ordering 
    $$ \mu_1(M)\le \mu_2(M)\le \dots \le \mu_p(M).$$
The following theorem is the main result of this work. As in the previous section, $(\alpha_k)_{k\in\mathbb{N}}$ is the increasing series of Dirichlet eigenvalues of the domain $U$ that we defined in \eqref{eq:ewp}. 

	\begin{theorem}
		\label{thm-BifComp}
		Assume that the assumptions (A1)-(A4) hold for the boundary value problem \eqref{eq:pdeGeneral} and that the linearised equation \eqref{eq:pdelinearised} only has the trivial solution for $\lambda=0$ and $\lambda=1$. 
		\begin{itemize}
			\item [(i)] 
				If $C_{1,0}\ge C_{1,1},\, C_{2,0}\ge C_{2,1}$ and there exists $j\in\{1,\dots,p_1\}$ such that $$\mu_j(C_{1,1})< \alpha_k < \mu_j(C_{1,0})\ \text{for some }k\in \mathbb{N}$$ or there exists $j\in\{1,\dots,p_2\}$ such that $$\mu_j(C_{2,1})< -\alpha_k < \mu_j(C_{2,0})\ \text{for some }k\in \mathbb{N},$$
				then there is a bifurcation point for \eqref{eq:pdeGeneral}.
			\item [(ii)] 
				If $D_{1,1}\ge D_{1,0},\, D_{2,1}\ge D_{2,0}$ and there exists $j\in \{1,\dots,p_1\}$ such that $$\mu_j(D_{1,0})< \alpha_k < \mu_j(D_{1,1})\ \text{for some }k\in \mathbb{N}$$ or there exists $j\in\{1,\dots,p_2\}$ such that $$\mu_j(D_{2,0})< -\alpha_k < \mu_j(D_{2,1})\ \text{for some }k\in \mathbb{N},$$
				then there is a bifurcation point for \eqref{eq:pdeGeneral}.
		\end{itemize}
	\end{theorem}
	\begin{proof}
		Since 
		$$\left\langle (L_\lambda-M_\lambda) u,v\right\rangle _H = \int_U \left\langle (C_\lambda -B_\lambda(x))u(x),v(x)\right\rangle \text{d}x,\ u,v\in H$$ and 
		$$\left\langle (L_\lambda-N_\lambda) u,v\right\rangle _H = \int_U \left\langle (D_\lambda -B_\lambda(x))u(x),v(x)\right\rangle \text{d}x,\ u,v\in H,$$ 
		we note that $L_\lambda-M_\lambda$ and $L_\lambda-N_\lambda$ are compact for all $\lambda \in I$ (see \eqref{eq:OperatorK}). Furthermore, as $B_0(x)-C_0,\, C_1-B_1(x)$ and $D_0-B_0(x),\,B_1(x)-D_1$ are positive semi-definite for all $x\in U$ by assumption, it follows that $N_0\le L_0\le M_0 $ and $M_1\le L_1\le N_1 $. Thus we obtain from \Cref{thm-comparison} that 
		$$\sfl(M)\le \sfl(L)\le \sfl(N).$$
		If we now can prove that $\sfl(M)>0 $ under the assumptions of $(i)$, and $\sfl(N)<0$ under the assumptions of $(ii)$, then the assertion follows from Theorem \ref{thm-bif}.\newline
		We only prove the assertion for $M$, as the argument for $N$ is very similar. If we apply the results of Section 4 to the matrices $C_\lambda$, we obtain that there is a decomposition of $H$ into finite dimensional subspaces $H_k$, $k\in \mathbb{N}$, such that each $H_k$ reduces $M_\lambda$ and the corresponding restrictions $M_\lambda ^k$ are represented by the matrices 
		\begin{align*}
			M_\lambda^k = 
			\begin{pmatrix}
				I_{p_1\times p_1} &  & 0 \\
				0 &  & -I_{p_2\times p_2}
			\end{pmatrix}
			-\frac{1}{\alpha_k}C_{\lambda},\, \lambda \in I.
		\end{align*}
		Now it follows from Theorem \ref{thm:bifurctaion1} that 
		\begin{align}
			\label{eq: sf-pathM}
			\begin{split}
			\sfl(M)&=i(C_1)-i(C_0)\\
			&= \frac{1}{2} \sum_{k=1}^{\infty} (\sgn(M_1^k)-\sgn(-A))-\frac{1}{2}\sum_{k=1}^\infty (\sgn(M_0^k)-\sgn(-A))\\
			&=\frac{1}{2} \sum_{k=1}^{\infty} (\sgn(M_1^k)-\sgn(M_0^k))\\
			&=\sum_{k=1}^{\infty} \left(\mu_{\text{Morse}}(M^k_0)-\mu_{\text{Morse}}(M^k_1)\right).
			\end{split}
		\end{align}
		Since the matrices 
			$\begin{pmatrix}
			I_{p_1\times p_1} &  & 0 \\
			0 &  & -I_{p_2\times p_2}
			\end{pmatrix}$ 
		and $-\frac{1}{\alpha_k}C_{\lambda}$ commute for all $\lambda \in I$, the eigenvalues of $M_\lambda^k$ are the sum of the eigenvalues of 			
			$\begin{pmatrix}
			I_{p_1\times p_1} &  & 0 \\
			0 &  & -I_{p_2\times p_2}
			\end{pmatrix}$ 
		and $-\frac{1}{\alpha_k}C_{\lambda}$.
		Moreover, note that
		$$\sigma(C_\lambda)=\sigma(C_{1,\lambda}) \cup \sigma(C_{2,\lambda}).$$
		Now, for any $k\in\mathbb{N}$, $1-\frac{\mu_i(C_{1,\lambda})}{\alpha_k}<0$ if and only if $\alpha_k< \mu_i(C_{1,\lambda})
		$  for  $i\in \{1,\dots,p_1\}$, and $-1-\frac{\mu_i(C_{2,\lambda})}{\alpha_k}<0$ if and only if $-\alpha_k< \mu_i(C_{2,\lambda})$ for  $i\in\{1,\dots,p_2\}$. Thus we obtain 
		$$\mu_{\text{Morse}}(M_\lambda ^k)=\vert (\alpha_k,\infty)\cap \sigma (C_{1,\lambda})\vert +\vert (-\alpha_k,\infty)\cap \sigma (C_{2,\lambda})\vert,$$
where here $|\cdot|$ stands for the cardinality of a set. Plugging this into \eqref{eq: sf-pathM} yields 
		\begin{align}
			\label{eq: sf-pathM2}
			\begin{split}
			\sfl(M)=&\sum_{k=1}^\infty (\vert (\alpha_k,\infty)\cap \sigma (C_{1,0})\vert +\vert (-\alpha_k,\infty)\cap \sigma (C_{2,0})\vert \\& -\vert (\alpha_k,\infty)\cap \sigma (C_{1,1})\vert -\vert (-\alpha_k,\infty)\cap \sigma (C_{2,1})\vert).
			\end{split}
		\end{align}
		Since $C_{1,0}\ge C_{1,1}$ and $C_{2,0}\ge C_{2,1}$, the matrices $C_{1,0}-C_{1,1}$ and $C_{2,0}-C_{2,1}$ are positive semi-definite. Thus it follows from  Weyl's inequality for symmetric matrices that
		\begin{align*}
			\mu_i(C_{1,1})\le \mu_i(C_{1,1}+(C_{1,0}-C_{1,1}))=\mu_i(C_{1,0})\ \text{for all}\ i\in\{1,\dots,p_1\}
		\end{align*}
		and 
		\begin{align*}
			\mu_i(C_{2,1})\le \mu_i(C_{2,1}+(C_{2,0}-C_{2,1}))=\mu_i(C_{2,0})\ \text{for all}\ i\in\{1,\dots,p_2\}.
		\end{align*}
		Consequently, for each $k\in\mathbb{N}$,
		\begin{align*}
			\vert (\alpha_k,\infty)\cap \sigma (C_{1,0})\vert&\geq \vert (\alpha_k,\infty)\cap \sigma (C_{1,1})\vert\\
			\vert (-\alpha_k,\infty)\cap \sigma (C_{2,0})\vert&\geq\vert (-\alpha_k,\infty)\cap \sigma (C_{2,1})\vert)
		\end{align*}
		as $(\alpha_k)_{k\in\mathbb{N}}$ is an increasing sequence. Thus we see from \eqref{eq: sf-pathM2} that $\sfl(M)\geq 0$, and that actually $\sfl(M)>0$ if there exists $j\in\{1,\dots,p_1\}$ such that $$\mu_j(C_{1,1})< \alpha_k < \mu_j(C_{1,0})\, \text{for some }k\in \mathbb{N}$$ or there exists $j\in\{1,\dots,p_2\}$ such that $$\mu_j(C_{2,1})< -\alpha_k < \mu_j(C_{2,0})\, \text{for some }k\in \mathbb{N}.$$
This finishes the proof.
	\end{proof} 
\noindent
Note that the assumptions of Theorem \ref{thm-BifComp} imply that either $B_0(x)\ge B_1(x)$ or $B_1(x)\ge B_0(x)$ for all $x\in U$. Let us also point out that the proof of Theorem \ref{thm-BifComp} yields the following spectral flow formula in the setting of Section \ref{sec:4}, which significantly simplifies the index formula in Theorem \ref{thm:bifurctaion1} if the matrices $B_\lambda$ are block diagonal.

\begin{cor}
If under the assumptions of Theorem \ref{thm:bifurctaion1}, the matrices $B_\lambda$ are of the form

\[B_\lambda=\begin{pmatrix}
		C_{1,\lambda} & 0 \\ 
		0 & C_{2,\lambda}
		\end{pmatrix}\]
for symmetric matrices $C_{1,\lambda}\in\Mat(p_1,\mathbb{R})$ and $C_{2,\lambda}\in\Mat(p_2,\mathbb{R})$, then

\begin{align}
				\nonumber
				\sfl(L)=&\sum_{k=1}^\infty (\vert (\alpha_k,\infty)\cap \sigma (C_{1,0})\vert +\vert (-\alpha_k,\infty)\cap \sigma (C_{2,0})\vert \\
				\nonumber& -\vert (\alpha_k,\infty)\cap \sigma (C_{1,1})\vert -\vert (-\alpha_k,\infty)\cap \sigma (C_{2,1})\vert)
			\end{align}

\end{cor} 
\noindent
We now set for $\lambda \in I$			
		
		$$ \gamma_\lambda:=\underset{x\in U}{\sup} \{\mu_1(B_\lambda(x)),...,\mu_p(B_\lambda(x))\},$$
			$$ \beta_\lambda:=\underset{x\in U}{\inf} \{\mu_1(B_\lambda(x)),...,\mu_p(B_\lambda(x))\},$$
and point out that the following simple corollary of Theorem \ref{thm-BifComp} is the best possible result that can be obtained by the comparison methods in \cite{BifJac} and \cite{Edinburgh}.

\begin{cor}		
	\label{cor:eigcomp}		
Under the assumptions of Theorem \ref{thm-BifComp},	
			
			\begin{enumerate}
				\item[(i)]
				if $\beta_0> \gamma_1$ and there exists $k\in \mathbb{N}$ such that
				$$ \gamma_1 <\alpha_k<\beta_0\ \text{or}\ \gamma_1<-\alpha_k<\beta_0$$
				then there is a bifurcation point for \eqref{eq:pdeGeneral}, 
				\item[(ii)]
				If $\beta_1> \gamma_0$ and there exists $k\in \mathbb{N}$ such that
				$$ \gamma_0 <\alpha_k<\beta_1\ \text{or}\ \gamma_0<-\alpha_k<\beta_1$$
				then there is a bifurcation point for \eqref{eq:pdeGeneral}. 
			\end{enumerate}
\end{cor}

\begin{proof}
We just set $C_\lambda =(\beta_0+\lambda(\gamma_1-\beta_0))I_{p\times p}$ and  $D_\lambda:=(\gamma_0+\lambda(\beta_1-\gamma_0))I_{p\times p}$ in \Cref{thm-BifComp}. 
\end{proof}			
\noindent
We now provide an example which shows that our main Theorem \ref{thm-bif} indeed is stronger than the previously known comparison methods from \cite{BifJac} and \cite{Edinburgh}. Consider the system \eqref{eq:pdeIntro} in the special case that $U=(0,\pi)$ and $p_1=1=p_2$, i.e. the systems \eqref{eq:pdeIntro} are of the form 
 	\begin{align}
 		\label{eq:pdeIntro2}
 		\left\{
 		\begin{array}{rl}
 			- u''_1(x) &= \frac{\partial F}{\partial u_1}(\lambda,u(x))\ \hspace*{0.25cm} \mathrm{in}\ (0,\pi)\\		
 			 u''_2(x) &= \frac{\partial F}{\partial u_2}(\lambda,u(x))\ \hspace*{0.25cm} \mathrm{in}\ (0,\pi)\\	
 			u(0) &= u(\pi)=0,
 		\end{array}
 		\right.
 	\end{align}
 	and we also assume that 
 	$$B_0:=\nabla_u^2 F(0,0)=
 	\begin{pmatrix}
 		8 & -2 \\
 		-2 & 5
 	\end{pmatrix} \text{ and } B_1:=\nabla_u^2F(1,0)=
 	\begin{pmatrix}
 	-3 & 1 \\
 	1 & 2
	\end{pmatrix}.$$
 Note that the Dirichlet eigenvalues of the domain $U=(0,\pi)$ are $\alpha_k=k^2.$  
	We now choose $C_\lambda=(\beta_0+\lambda(\gamma_1-\beta_0))I_{2\times 2}$ as in Corollary \ref{cor:eigcomp}, where $\beta_0=4$ is the smallest eigenvalue of $B_0$ and $\gamma_1=\frac{\sqrt{29}-1}{2}$ is the largest eigenvalue of $B_1$.
	Now there is no $k\in \mathbb{N}$ such that $\gamma_1<\alpha_k< \beta_0$, and thus we cannot use \Cref{cor:eigcomp} to investigate bifurcation points for \eqref{eq:pdeIntro2}. 
	However, if $$C_0=\begin{pmatrix}
		C_{1,0} & 0 \\ 
		0 & C_{2,0}
	\end{pmatrix}=\begin{pmatrix}
		5 & 0 \\
		0 & 3
	\end{pmatrix} \text{ and } C_1=	\begin{pmatrix}
	C_{1,1} & 0 \\ 
	0 & C_{2,1}
\end{pmatrix}=\begin{pmatrix}
	1 & 0 \\
	0 & 3
\end{pmatrix}$$ in \Cref{thm-BifComp}, we get the existence of a bifurcation point for \eqref{eq:pdeIntro2} as 
$$\mu(C_{1,1})=1<\alpha_2=4<5=\mu(C_{1,0}).$$
We finish this section by considering for later applications the special case that $p=2$ and $p_1=1=p_2$, i.e. the systems \eqref{eq:pdeGeneral} are of the form 
		\begin{align}
			\label{eq:pdeGeneral2}
			\left\{
			\begin{array}{rl}
				-\Delta u_1(x) &= \nabla_{u_1} F(\lambda,x,u(x))\ \hspace*{0.25cm} \mathrm{in}\ U\\		
				\Delta u_2(x) &= \nabla_{u_2} F(\lambda,x,u(x))\ \hspace*{0.25cm} \mathrm{in}\ U\\	
				u(x) &= 0\ \hspace*{2.5cm}  \mathrm{on}\ \partial U,
			\end{array}
			\right.
		\end{align}
		where now again $U\subseteq \mathbb{R}^N$ is a bounded smooth domain for some $N\in \mathbb{N}$.
		According to assumption $(A2)$, we set 
		$$B_\lambda(x):=
		\begin{pmatrix}
			b_{11}(\lambda,x)& b_{12}(\lambda,x)  \\
			b_{12}(\lambda,x)	& b_{22}(\lambda,x) 
		\end{pmatrix}$$
		and assume that the corresponding linearised system \eqref{eq:pdelinearised2} with $A=\diag(-1,1)$ has only the trivial solution for $\lambda=0$ and $\lambda=1$.
		If we now set 
		$C_\lambda:=
		\begin{pmatrix}
			c_{1\lambda}& 0  \\
			0	& c_{2\lambda} 
		\end{pmatrix}\in \Mat(2,\mathbb{R})$ 
		and 
		$D_\lambda:=
		\begin{pmatrix}
			d_{1\lambda}& 0  \\
			0	& d_{2\lambda} 
		\end{pmatrix}\in \Mat(2,\mathbb{R})$,
		we can apply \Cref{thm-BifComp} and immediately get that there is a bifurcation point for \eqref{eq:pdeGeneral2} under any of the following assumptions: 
		\begin{enumerate}
			\item[(i)]
			there are $c_{10},c_{11},c_{20},c_{21}\in \mathbb{R}$ such that
			
			$$ \underset{x\in \overline{U}}{\min}\ b_{22}(0,x)-\underset{x\in \overline{U}}{\max}\ \vert b_{12}(0,x) \vert \ge c_{20} \ge c_{21}\ge  \underset{x\in \overline{U}}{\max}\ b_{22}(1,x)+\underset{x\in \overline{U}}{\max}\ \vert b_{12}(1,x) \vert $$
			
			and there exists $k\in \mathbb{N}$ such that
			\begin{align}
				\nonumber
				\underset{x\in \overline{U}}{\min}\ b_{11}(0,x)-\underset{x\in \overline{U}}{\max}\ \vert b_{12}(0,x) \vert \ge c_{10} > \alpha_k > c_{11}\ge   \underset{x\in \overline{U}}{\max}\ b_{11}(1,x)+\underset{x\in \overline{U}}{\max}\ \vert b_{12}(1,x) \vert. 
			\end{align}
			\item[(ii)] 
			there are $c_{10},c_{11},c_{20},c_{21}\in \mathbb{R}$ such that 
			
			$$ \underset{x\in \overline{U}}{\max}\ b_{11}(0,x)-\underset{x\in \overline{U}}{\max}\ \vert b_{12}(0,x) \vert \ge c_{10} \ge c_{11}\ge  \underset{x\in \overline{U}}{\max}\ b_{11}(1,x)+\underset{x\in \overline{U}}{\max}\ \vert b_{12}(1,x) \vert $$
			
			and there exists $k\in \mathbb{N}$ such that
			\begin{align}
				\nonumber
				\underset{x\in \overline{U}}{\min}\ b_{22}(0,x)-\underset{x\in \overline{U}}{\max}\ \vert b_{12}(0,x) \vert \ge c_{20} > -\alpha_k>  c_{21} \ge	\underset{x\in \overline{U}}{\max}\ b_{22}(1,x)+\underset{x\in \overline{U}}{\max}\ \vert b_{12}(1,x) \vert. 
			\end{align}
			\item[(iii)] 
			there are $d_{10},d_{11},d_{20},d_{21}\in \mathbb{R}$ such that 
			
			$$ \underset{x\in \overline{U}}{\min}\ b_{22}(1,x)-\underset{x\in \overline{U}}{\max}\ \vert b_{12}(1,x) \vert \ge d_{21}\ge d_{20}\ge  \underset{x\in \overline{U}}{\max}\ b_{22}(0,x)+\underset{x\in \overline{U}}{\max}\ \vert b_{12}(0,x) \vert $$
			
		 and there exists $k\in \mathbb{N}$ such that
			\begin{align}
				\nonumber
				\underset{x\in \overline{U}}{\min}\ b_{11}(1,x)-\underset{x\in \overline{U}}{\max}\ \vert b_{12}(1,x) \vert \ge d_{11} > \alpha_k>d_{10}
				\ge \underset{x\in \overline{U}}{\max}\ b_{11}(0,x)+\underset{x\in \overline{U}}{\max}\ \vert b_{12}(0,x) \vert. 
			\end{align}
			\item[(iv)] 
			there are $d_{10},d_{11},d_{20},d_{21}\in \mathbb{R}$ such that 
			
			$$ \underset{x\in \overline{U}}{\min}\ b_{11}(1,x)-\underset{x\in \overline{U}}{\max}\ \vert b_{12}(1,x) \vert \ge d_{11} \ge d_{10}\ge  \underset{x\in \overline{U}}{\max}\ b_{11}(0,x)+\underset{x\in \overline{U}}{\max}\ \vert b_{12}(0,x) \vert $$
			
			and there exists $k\in \mathbb{N}$ such that
			\begin{align}
				\nonumber
				\underset{x\in \overline{U}}{\min}\ b_{22}(1,x)-\underset{x\in \overline{U}}{\max}\ \vert b_{12}(1,x) \vert \ge d_{21} > -\alpha_k> d_{20} 
				\ge \underset{x\in \overline{U}}{\max}\ b_{22}(0,x)+\underset{x\in \overline{U}}{\max}\ \vert b_{12}(0,x) \vert.  
			\end{align}	 
		\end{enumerate}
 
%%%%%%%%%%%%%%%%%%%%%%%%%%%%%%%%%%%%%%%%%%%%%%%%%%%%%%%%%%%%%%%%%%%%%%%%%%%%%%%%%%%%%%%%%%%%%%%%%%%%%%%%%%%%%%%%%%%%%%%%%%%%%%%%%%%%%%%%%%%%%%%%%%%%%%%%%%%%%%%%%%%%%%%

\section{Application to Bifurcation on Shrinking Domains}
The aim of this section is to apply the results of the previous sections to a setting that has been investigated in \cite{AleIchDomain}, \cite{AleBall}, \cite{ProcHan} and in particular in \cite{AleSmaleIndef}. We consider the system 

\begin{align}
\label{eq:pdeshrinking}
		\left\{
		\begin{array}{rl}
			A\Delta u(x) &= \nabla_u F(x,u(x))\ \hspace*{0.25cm} \mathrm{in}\ U\\		
			u(x) &= 0\ \hspace*{2cm}  \mathrm{on}\ \partial U,
		\end{array}
		\right.
\end{align}
that does not depend on a parameter $\lambda$, and where now $U\subset\mathbb{R}^N$ is a bounded smooth domain that is star-shaped with respect to $0$. The matrix $A$ is as in $(A4)$ in Section \ref{sec:VarSet}, and we assume that $0$ is a critical point of $F(x,\cdot):\mathbb{R}^p\rightarrow\mathbb{R}$ for all $x\in U$ and henceforth set 
	
	\begin{align}\label{Bshrunk}
	B(x):=\nabla^2_uF(x,0).
	\end{align}
	Below we will also once again need the linearised system
	
	\begin{align}
\label{eq:pdeshrinkinglin}
		\left\{
		\begin{array}{rl}
			A\Delta u(x) &= B(x)u(x)\ \hspace*{0.25cm} \mathrm{in}\ U\\		
			u(x) &= 0\ \hspace*{1.5cm}  \mathrm{on}\ \partial U.
		\end{array}
		\right.
\end{align}
	Our final standing assumption is as in $(A3)$
	
	\begin{itemize}	
\item		 There exist $C>0$ and $1\le s< (N+2)(N-2)^{-1}$ if $N\geq 3$ such that 
		$$\vert\nabla^2_u F(x,u)\vert\le C(1+\vert u\vert^{s-1})\,.$$ 
		If $N=2$, we instead require that $s\in [1,\infty)$, and if $N=1$ we do not impose any growth condition.
		\end{itemize}
As $U$ is star-shaped with respect to $0$, it makes sense to consider \eqref{eq:pdeshrinking} on the shrunk domains
	
	\[U_r:=\{r\, x:\, x\in U\},\quad r\in(0,1],\]
i.e.

\begin{align}
\label{eq:pdeshrinked}
		\left\{
		\begin{array}{rl}
			A\Delta u(x) &= \nabla_u F(x,u(x))\ \hspace*{0.25cm} \mathrm{in}\ U_r\\		
			u(x) &= 0\ \hspace*{2.0cm}  \mathrm{on}\ \partial U_r.
		\end{array}
		\right.
\end{align}
Note that by assumption, the constant function $u\equiv 0$ is a solution of \eqref{eq:pdeshrinked} for all $0<r\leq 1$.
We call $r^\ast\in[0,1]$ a bifurcation radius for \eqref{eq:pdeshrinking} if there are sequences $r_n\in(0,1]$ and $u_n\in H^1_0(U_r,\mathbb{R}^p)$, $u_n\neq 0$, $n\in\mathbb{N}$, such that $r_n\rightarrow r^\ast$ and $\|u_n\|_{H^1_0(U_r,\mathbb{R}^p)}\rightarrow 0$ as $n\rightarrow\infty$.\\
As in \cite{AleSmaleIndef} it follows that for any fixed $r\in(0,1]$ the solutions of \eqref{eq:pdeshrinked} are the critical points of the functional

\begin{align*}
		\tilde{f}(u) :=\frac{1}{2}\int_{U_r} \sum_{i=1}^p (-a_i\vert\nabla u_i(x)\vert^2)\,\text{d}x-\int_{U_r} F(x, u(x))\,\text{d}x,
	\end{align*}
which after a change of coordinates $x\mapsto r\cdot x$ transforms to

\begin{align}
		\label{eq:functionalshrinked}
		f(r,u) :=\frac{1}{2}\int_{U} \sum_{i=1}^p (-a_i\vert\nabla u_i(x)\vert^2)\,\text{d}x-r^2\int_{U} F(r\cdot x, u(x))\,\text{d}x,
	\end{align}
and which now can be considered for $r\in[0,1]$. Thus we are now in the setting of Section \ref{sec:VarSet} and consequently the second derivative of $f_r:H^1_0(U,\mathbb{R}^p)\rightarrow\mathbb{R}$ is given by

\begin{align}\label{shrinkingL}
\begin{split}
D^2_0f_r(u,v)&=\int_U \sum_{i=1}^p (-a_i\langle \nabla u_i(x),\nabla v_i(x)\rangle) \,\text{d}x -r^2\int_U \left\langle B (r\cdot x)u(x),v(x)\right\rangle\,\text{d}x\\
&=\langle(T+K_r)u,v\rangle_{H^1_0(U,\mathbb{R}^p)}.
\end{split}
\end{align}
Note that the Hessian $L_r=T+K_r$ is invertible for $r=0$ as $K_0=0$. By the implicit function theorem, it in particular follows that there are no bifurcation radii close to $0$.\\
Let us now first consider the case that $B$ in \eqref{Bshrunk} does not explicitly depend on $x$ as in Section \ref{sec:4}. It is an immediate consequence of Theorem \ref{thm:bifurctaion1} that there is a bifurcation radius $r^\ast\in(0,1]$ for \eqref{eq:pdeshrinked} if \eqref{eq:pdeshrinkinglin} has no non-trivial solution and

\begin{align}\label{indexshrink}
\sum_{k=1}^\infty (\sgn(L^k)-\sgn(-A))\neq 0,
\end{align}
where 

\begin{align}\label{Lkshrink}
		L^k = 
		\begin{pmatrix}
			I_{p_1\times p_1} &  & 0 \\
			0 &  & -I_{p_2\times p_2}
		\end{pmatrix}
		-\frac{1}{\alpha_k}B=-A-\frac{1}{\alpha_k}B,
	\end{align}
	and $(\alpha_k)$ is the sequence of Dirichlet eigenvalues of the domain $U$. Thus the existence of bifurcation radii can be obtained by relating the eigenvalues of $B$ to the Dirichlet eigenvalues of the domain $U$. The following theorem shows that it is not even necessary to compute \eqref{indexshrink} under reasonable assumptions. Let us emphasize that in the second part of the theorem, we can even lift the assumption that the linearised equation \eqref{eq:pdeshrinkinglin} has no non-trivial solution, and thus the existence of bifurcation radii can be obtained solely by estimating the smallest eigenvalue of the matrix $B$. As before we denote by $\mu_1(M)\leq\ldots\leq\mu_p(M)$ the eigenvalues of a symmetric matrix $M\in\Mat(p,\mathbb{R})$.  

\begin{theorem}\label{thm:shrink}
If the matrix family $B$ in \eqref{Bshrunk} does not depend on $x\in U$ and $1\leq p_1,p_2\leq p-1$, then there is a bifurcation radius $r^\ast\in(0,1)$ for \eqref{eq:pdeshrinked} if

\begin{enumerate}
\item[(i)] either \eqref{eq:pdeshrinkinglin} has no non-trivial solution and

\[-\alpha_1<\mu_1(B)\quad\text{as well as}\quad\alpha_1<\mu_p(B),\]
\item[(ii)] or $\mu_1(B)>\alpha_1$, i.e., the smallest eigenvalue of $B$ is larger than the smallest Dirichlet eigenvalue of the domain $U$.

\end{enumerate} 
\end{theorem}

\begin{proof}
We begin by showing (i), which is a consequence of Theorem \ref{thm:bifurctaion1} and Corollary \ref{cor-bif}. Thus we aim to show that the index in \eqref{indexshrink} does not vanish under the given assumptions.\\
As $\sgn(M)=p-2\mu_{Morse}(M)$ for any symmetric and invertible $M\in\Mat(p,\mathbb{R})$, we obtain

\[\sgn(L^k)=p-2\mu_{Morse}(L^k),\qquad \sgn(-A)=p_1+p_2-2p_1=p_1-p_2,\]
and thus in \eqref{indexshrink}

\begin{align}\label{indexreformulated}
\sgn(L^k)-\sgn(-A)=p_1+p_2-2\mu_{Morse}(L^k)-p_1+p_2=2p_2-2\mu_{Morse}(L^k).
\end{align}
By Weyl's inequality on the perturbation of eigenvalues, we obtain in \eqref{Lkshrink} for $1\leq l\leq p$ and $l\leq i\leq p$

\[\mu_l(L^k)\leq\mu_i(-A)+\mu_{p+l-i}(-\frac{1}{\alpha_k}B),\]
and thus for $1\leq l\leq p_1$

\begin{align*}
\mu_l(L^k)\leq 1+\mu_{p+l-i}(-\frac{1}{\alpha_k}B)=1-\frac{1}{\alpha_k}\mu_{i-l+1}(B),
\end{align*}
which is negative if and only if $\alpha_k<\mu_{i-l+1}(B)$ for some $l\leq i\leq p$. In particular, we note that $\mu_1(L^k)$ is negative if $\alpha_k<\mu_p(B)$. Now we consider the case that $p_1+1\leq l\leq p$. By Weyl's inequality

\begin{align*}
\mu_l(L^k)\leq -1+\mu_{p+l-i}(-\frac{1}{\alpha_k}B)=-1-\frac{1}{\alpha_k}\mu_{i-l+1}(B),
\end{align*}
which is negative if and only if $-\alpha_k<\mu_{i-l+1}(B)$ for some $l\leq i\leq p$.\\ 
In summary, as the sequence $(\alpha_k)_{k\in\mathbb{N}}$ is ascending, if $\mu_1(B)>-\alpha_1$, then $\mu_l(L^k)<0$ for all $p_1+1\leq l\leq p$ and thus $\mu_{Morse}(L^k)\geq p-(p_1+1)=p_2$ for all $k\in\mathbb{N}$. If moreover, $\alpha_1<\mu_p(B)$, then $\mu_{Morse}(L^1)\geq p_2+1$ and we obtain in \eqref{indexshrink} and by \eqref{indexreformulated} that

\begin{align*}
\sum_{k=1}^\infty (\sgn(L^k)-\sgn(-A))&=2\sum_{k=1}^\infty (p_2-\mu_{Morse}(L^k))\leq 2(p_2-p_2-1)+2\sum_{k=2}^\infty (p_2-\mu_{Morse}(L^k))\\
&=-2+2\sum_{k=2}^\infty (p_2-\mu_{Morse}(L^k))
\end{align*}
where the terms in the latter sum are all non-positive. Thus $(i)$ is proved.\\
Let us now consider $(ii)$ and assume in a first step that \eqref{eq:pdeshrinkinglin} has no non-trivial solution, which again implies that the matrices $L^k$, $k\in\mathbb{N}$, in \eqref{Lkshrink} are invertible. As in \eqref{indexreformulated}, we note that we need to show that

\begin{align}\label{indexreformulated2}
\sum^\infty_{k=1}(p_2-\mu_{Morse}(L^k))\neq 0.
\end{align}
By using Weyl's inequality once again, we obtain that

\[\mu_l(L^k)\leq\mu_l(-A)+\mu_p(-\frac{1}{\alpha_k}B)=\mu_l(-A)-\frac{1}{\alpha_k}\mu_1(B).\]
Thus, if $1\leq l\leq p_1$, we see that $\mu_l(L^k)<0$ if $1-\frac{1}{\alpha_k}\mu_1(B)<0$ which is the case if $\alpha_k<\mu_1(B)$. Moreover, if $p_1+1\leq l\leq p$, then $\mu_l(L^k)<0$ if $-1-\frac{1}{\alpha_k}\mu_1(B)<0$ which is equivalent to $-\alpha_k<\mu_1(B)$. Now let us assume as in the theorem that $\mu_1(B)>\alpha_1$. Then $\mu_1(B)>-\alpha_k$ and thus $\mu_l(B)>0$ for $p_1+1\leq l\leq p$, which shows that $\mu_{Morse}(L^k)\geq p_2$ for all $k\in\mathbb{N}$. Moreover, the first case from above tells us that $\mu_1(L^1)<0$ and consequently $\mu_{Morse}(L^1)\geq p_2+1$. As in $(i)$, we obtain in \eqref{indexreformulated2} that

\[\sum^\infty_{k=1}(p_2-\mu_{Morse}(L^k))=p_2-(p_2+1)+\sum^\infty_{k=2}(p_2-\mu_{Morse}(L^k))=-1+\sum^\infty_{k=2}(p_2-\mu_{Morse}(L^k)),\]
where the sum on the right hand side is non-positive. Thus there is a bifurcation radius by Theorem \ref{thm:bifurctaion1}.\\
Now our aim is to lift the assumption that \eqref{eq:pdeshrinkinglin} has no non-trivial solution, for which we need to make a brief digression that stems from a method for computing the spectral flow that was introduced by Robbin and Salamon in \cite{Robbin-Salamon} and generalised in \cite{Homoclinics}.\\
Let $H$ be a Hilbert space and $\mathcal{L}=\{\mathcal{L}_\lambda\}_{\lambda\in I}$ a path in the normed space of bounded operators $\mathcal{L}(H)$, which we assume to be continuously differentiable with respect to the operator norm. We also assume that each $\mathcal{L}_\lambda$ is selfadjoint and Fredholm. A parameter value $\lambda^\ast$ is called a \textit{crossing} of $\mathcal{L}$ if $\ker(\mathcal{L}_{\lambda^\ast})\neq \{0\}$, and the \textit{crossing form} of a crossing is the quadratic form defined by

\[\Gamma(\mathcal{L},\lambda^\ast)[u]=\langle \dot{\mathcal{L}}_{\lambda^\ast}u,u\rangle,\qquad u\in \ker(\mathcal{L}_{\lambda^\ast}),\]   
where $\dot{\mathcal{L}}_{\lambda^\ast}$ denotes the derivative of the path $\mathcal{L}$ at $\lambda=\lambda^\ast$. A crossing is called \textit{regular} if $\Gamma(\mathcal{L},\lambda^\ast)$ is non-degenerate. It was shown in \cite{Robbin-Salamon} that crossings are isolated if they are regular, i.e., if $\lambda^\ast$ is regular, then $\mathcal{L}_\lambda$ is invertible for any $\lambda\neq\lambda^\ast$ that is sufficiently close to $\lambda$.\\
In our case, it follows from \eqref{shrinkingL} and the assumption that $B$ does not depend on $x$, that the crossing form is

\[\Gamma(L,1)[u]=-2\int_U{\langle Bu(x),u(x)\rangle\, dx},\quad u\in\ker(L_{1}),\] 
where $\ker(L_1)$ is just the space of solutions of the linearised equation \eqref{eq:pdeshrinkinglin}. Now, since $\mu_1(B)>0$, $B$ is positive definite and consequently $\Gamma(L,1)$ is negative definite and in particular non-degenerate. Thus $L_r$ is invertible for any $r\neq 1$ that is sufficiently close to $1$ and we can apply the first part of the proof to the equations on a slightly shrunk domain, where $\mu_1(B)$ is still larger than the smallest Dirichlet eigenvalue. This yields the claimed bifurcation even though \eqref{eq:pdeshrinkinglin} might have non-trivial solutions and finishes the proof of $(ii)$.  
\end{proof}
\noindent
We now drop the assumption that $B$ in \eqref{Bshrunk} does not depend on $x\in U$ and obtain bifurcation results by the comparison theory that we developed in Section \ref{Sec:Comp}. To state the results as clear as possible, we consider a system of two equations as in \cite{Edinburgh}. Obviously, the reader will not have any difficulties in finding the corresponding conditions in the case of more coupled equations. To fix notations, let us state \eqref{eq:pdeshrinked} in this setting once again

\begin{align}\label{eq:shrinking2}
		\left\{
		\begin{array}{rl}
			-\Delta u(x) &= \frac{\partial F}{\partial u}(x,u(x),v(x))\ \hspace*{0.25cm} \mathrm{in}\ U_r\\
			\Delta v(x) &= \frac{\partial F}{\partial v}(x,u(x),v(x))\ \hspace*{0.25cm} \mathrm{in}\ U_r\\		
			u(x) &=v(x)= 0\ \hspace*{1.5cm}  \mathrm{on}\ \partial U_r,
		\end{array}
		\right.
\end{align}
and let us emphasize that $u,v$ are now scalar functions. We set

\begin{align}\label{B}
B(x):=
		\begin{pmatrix}
			b_{11}(x)& b_{12}(x)  \\
			b_{12}(x)	& b_{22}(x) 
		\end{pmatrix},
		\end{align}
where $B(x)=D^2_0F$ is the Hessian of $F:U\times\mathbb{R}^2\rightarrow\mathbb{R}$ at $0\in\mathbb{R}^2$. 
The following theorem underpins the strength of Theorem \ref{thm-BifComp} about comparison. Note that we no longer necessarily require that the linearised equation \eqref{eq:pdeshrinkinglin} has no non-trivial solution, which as in Theorem \ref{thm:bifurctaion1} allows to find bifurcation radii just by inspecting the matrix \eqref{B}.

\begin{theorem}\label{thm-bifshrink}
Let $p=2$ and $A=\diag(-1,1)$. If either

\begin{itemize}
 \item[(i)] $\underset{x\in \overline{U}}{min}\ b_{22}(x)> \underset{x\in \overline{U}}{max}\ |b_{12}(x)|$ and $\underset{x\in \overline{U}}{min}\ b_{11}(x)>\underset{x\in \overline{U}}{max}\ |b_{12}(x)|+\alpha_1$, as well as
\begin{align}\label{posdef}
\frac{\partial}{\partial r}\mid_{r=1}\langle B(rx)u,u\rangle\geq 0,\quad x\in\mathbb{R}^n,\quad u\in\mathbb{R}^2,
\end{align} 
\end{itemize}
or
\begin{itemize} 
 \item[(ii)] $\underset{x\in \overline{U}}{max}\ b_{11}(x)< -\underset{x\in \overline{U}}{max}\ |b_{12}(x)|$ and $\underset{x\in \overline{U}}{max}\ b_{22}(x)<-\underset{x\in \overline{U}}{max}\ |b_{12}(x)|-\alpha_1$, as well as
\begin{align}\label{negdef}
\frac{\partial}{\partial r}\mid_{r=1}\langle B(rx)u,u\rangle\leq 0,\quad x\in\mathbb{R}^n,\quad u\in\mathbb{R}^2,
\end{align} 
\end{itemize}
then there is a bifurcation radius for \eqref{eq:shrinking2}.
\end{theorem} 
\noindent

\begin{proof}
Firstly, if $(i)$ or $(ii)$ holds and \eqref{eq:pdeshrinkinglin} has no non-trivial solution, then the assertion is a simple consequence of our investigations about equation \eqref{eq:pdeGeneral2} at the end of Section \ref{Sec:Comp} and it is not even necessary to assume \eqref{posdef} or \eqref{negdef}. We will need this observation below.\\
If \eqref{eq:pdeshrinkinglin} has a non-trivial solution, and consequently $L_1$ has a non-trivial kernel, we consider crossing forms as in the proof of Theorem \ref{thm:shrink}. Here the crossing form at $r=1$ is given by 

\[\Gamma(L,1)[u]=-\frac{\partial}{\partial r}\mid_{r=1}r^2\int_U{\langle B(r\,x)u,u\rangle\, dx}=-2\int_U{\langle B(x)u,u\rangle\, dx}-\int_U{\frac{\partial}{\partial r}\mid_{r=1}\langle B(r\,x)u,u\rangle\, dx.}\]
Now note that under the assumptions of $(i)$, the matrix $B$ is positive definite, whereas in $(ii)$ it is negative definite. With the required estimates \eqref{posdef} and \eqref{negdef}, we see that $1$ is a regular crossing and thus isolated as in the proof of Theorem \ref{thm:shrink}. Now we just apply the first part of the proof on a slightly shrunk domain $U_{1-\varepsilon}$, where the first two estimates in $(i)$ or $(ii)$ still hold, but the corresponding operators $L_{1-\varepsilon}$ are invertible, i.e. the linearised equations \eqref{eq:pdeshrinkinglin} have no non-trivial solution on $U_{1-\varepsilon}$.
\end{proof}
\noindent
Let us once again emphasise that bifurcation raddii of \eqref{eq:pdeshrinking} were studied in \cite{AleIchDomain}, but no computable way to find them was provided. Note that Theorem \ref{thm-bifshrink} allows to find them only by knowledge of the coefficients in \eqref{eq:shrinking2}.\\
Let us finally provide a numerical example. We consider equation \eqref{eq:shrinking2} on a disc $U=D^2_r$ of radius $r>0$ in $\mathbb{R}^2$. It is well known that the Dirichlet eigenvalues of $D^2_r$ are given by

\begin{align}\label{Dirichelteigenvalues}
\lambda_{nm}=\left(\frac{\beta_{mn}}{r}\right)^2,\quad n\in\mathbb{N}\cup\{0\}, m\in\mathbb{N},
\end{align}
where $\beta_{nm}$ is the $m$-th positive zero of the $n$-th Bessel function. The smallest of these numbers is $\beta_{0,1}$ and its numerical value is approximately $2.40483$. Thus $5<\beta^2_{0,1}<6$ and we obtain from Theorem \ref{thm-bifshrink} that there is a bifurcation radius for any equation \eqref{eq:shrinking2} on a disc $D^2_r$ in $\mathbb{R}^2$ if either      

\[\underset{x\in \overline{D^2_r}}{max}\ b_{11}(x)< -\underset{x\in \overline{D^2_r}}{max}\ |b_{12}(x)|,\quad \underset{x\in \overline{D^2_r}}{max}\ b_{22}(x)<-\underset{x\in \overline{D^2_r}}{max}\ |b_{12}(x)|-\frac{6}{r^2},\quad\text{and}\quad\eqref{negdef}\]
hold, or 

\[\underset{x\in \overline{D^2_r}}{min}\ b_{22}(x)> \underset{x\in \overline{D^2_r}}{max}\ |b_{12}(x)|,\quad \underset{x\in \overline{D^2_r}}{min}\ b_{11}(x)>\underset{x\in \overline{D^2_r}}{max}\ |b_{12}(x)|+\frac{6}{r^2}\quad\text{and}\quad \eqref{posdef}.\]

%%%%%%%%%%%%%%%%%%%%%%%%%%%%%%%%%%%%%%%%%%%%%%%%%%%%%%%%%%%%%%%%%%%%%%%%%%%%%%%%%%%%%%%%%%%%%%%%%%%%%%%%%%%%%%%%%%%%%%%%%%%%%%%%%%%%%%%%%%%%%%%%%%%%%%%%%%%%%%%%%%%%%%%

	\subsubsection*{Acknowledgments}
	N. Waterstraat was supported by the Deutsche Forschungsgemeinschaft (DFG, German Research Foundation) - 459826435.

\vspace*{1.3cm}

\begin{minipage}{1.2\textwidth}
\begin{minipage}{0.4\textwidth}

Joanna Janczewska\\
Institute of Applied Mathematics\\
Faculty of Applied Physics and Mathematics\\
Gda\'{n}sk University of Technology\\
Narutowicza 11/12, 80-233 Gda\'{n}sk, Poland\\
joanna.janczewska@pg.edu.pl\\\\\\
Melanie M\"ockel\\
	Martin-Luther-Universit\"at Halle-Wittenberg\\
	Naturwissenschaftliche Fakult\"at II\\
	Institut f\"ur Mathematik\\
	06099 Halle (Saale), Germany\\
	melanie.moeckel@mathematik.uni-halle.de\\\\

\end{minipage}
\hfill
\begin{minipage}{0.6\textwidth}

Nils Waterstraat\\
Martin-Luther-Universit\"at Halle-Wittenberg\\
Naturwissenschaftliche Fakult\"at II\\
Institut f\"ur Mathematik\\
06099 Halle (Saale)\\
Germany\\
nils.waterstraat@mathematik.uni-halle.de
\end{minipage}
\end{minipage}

\end{document}